\documentclass[a4paper,11pt]{article}
\usepackage{amsfonts,amssymb,amsthm,cite,amsmath,amstext,amsbsy,bm}
\usepackage[toc,page]{appendix}
\usepackage{paralist}
\usepackage{authblk}
\usepackage{color}
\usepackage[colorlinks,linkcolor=red,citecolor=blue]{hyperref}
\usepackage{mathrsfs}
\usepackage{euscript}
\usepackage{mathtools}
\usepackage{dsfont}
\usepackage[all]{xy}

\DeclarePairedDelimiter\floor{\lfloor}{\rfloor}

\makeatletter
\newsavebox{\@brx}
\newcommand{\llangle}[1][]{\savebox{\@brx}{\(\m@th{#1\langle}\)}%
  \mathopen{\copy\@brx\kern-0.5\wd\@brx\usebox{\@brx}}}
\newcommand{\rrangle}[1][]{\savebox{\@brx}{\(\m@th{#1\rangle}\)}%
  \mathclose{\copy\@brx\kern-0.5\wd\@brx\usebox{\@brx}}}
\makeatother

\newenvironment{claim}[1]{\par\noindent\underline{Claim:}\space#1}{}
\newenvironment{claimproof}[1]{\par\noindent\underline{Proof:}\space#1}{\hfill $\blacksquare$}%creat claim environment

\numberwithin{equation}{section}
\definecolor{plum}{rgb}{0.8,0.2,0.8}

\title{\textbf{Effectivity in the Hyperbolicity-related Problems}}
\date{}
\author{Ya Deng\thanks{Email address: \texttt{Ya.Deng@univ-grenoble-alpes.fr}}}
\affil{Institut Fourier, Universit\'e Grenoble Alpes}

\begin{document}
\maketitle

\newtheorem{thm}{Theorem}[section]
\newenvironment{thmbis}[1]
{\renewcommand{\thethm}{\ref{#1}\emph{bis}}%
	\addtocounter{thm}{-1}%
	\begin{thm}}
	{\end{thm}}

\newtheorem{rem}{Remark}[section]
\newtheorem{problem}{Problem}[section]
\newtheorem{conjecture}{Conjecture}[section]
\newtheorem{lem}{Lemma}[section]
\newtheorem{cor}{Corollary}[section]
\newtheorem{dfn}{Definition}[section]
\newtheorem{proposition}{Proposition}[section]
\newtheorem{example}{Example}[section]
\theoremstyle{plain}
\newtheorem{main}{Main Theorem}

\def\oc{\mathcal{O}} \def\oce{\mathcal{O}_E} \def\xc{\mathcal{X}}
\def\ac{\mathcal{A}} \def\rc{\mathcal{R}} \def\mc{\mathcal{M}}
\def\wc{\mathcal{W}} \def\fc{\mathcal{F}} \def\cf{\mathcal{C_F^+}}
\def\jc{\mathcal{J}}
\def\ic{\mathcal{I}}
\def\kc{\mathcal{K}}
\def\lc{\mathcal{L}}
\def\vc{\mathcal{V}}
\def\yc{\mathcal{Y}}
\def\af{\mathbf{a}}

\def\cb{\mathbb{C}}
\def\ib{\mathbb{I}}
\def\ab{\mathbb{A}}
\def\vol{\operatorname{vol}}
\def\ord{\operatorname{ord}}
\def\Im{\operatorname{Im}}
\def\dm{\mathrm{d}}

\def\as{{a^\star}} \def\es{e^\star}

\def\tl{\widetilde} \def\tly{{Y}} \def\om{\omega}

\def\cb{\mathbb{C}} \def\nb{\mathbb{N}} \def\nbs{\mathbb{N}^\star}
\def\pb{\mathbb{P}} \def\pbe{\mathbb{P}(E)} \def\rb{\mathbb{R}}
\def\zbb{\mathbb{Z}}
\def\ys{\mathscr{Y}}
\def\ls{\mathscr{L}}
\def\hs{\mathscr{H}}
\def\grs{{\rm{Gr}}_{k+1}(V)}
\def\yf{\mathfrak{Y}}
\def\gf{\mathbf{G}}
\def\bff{\mathbf{b}}

\def\hb{\bold{H}} \def\fb{\bold{F}} \def\eb{\bold{E}}
\def\pbb{\bold{P}}

\def\nab{\overline{\nabla}} \def\n{|\!|} \def\spec{\textrm{Spec}\,}
\def\cinf{\mathcal{C}_\infty} \def\d{\partial}
\def\db{\overline{\partial}}
\def\hess{\sqrt{-1}\partial\overline{\partial}}
\def\zb{\overline{z}} \def\lra{\longrightarrow}
\def\tb{\overline{t}}
\def\sn{\sqrt{-1}}
\begin{abstract}
The aim of this work is to deal with effective questions related to the Kobayashi and Debarre conjectures, and based on the work of Damian Brotbek and Lionel Darondeau. We first show that if a line bundle $L$ generates $k$-jets, the $k$-th Wronskian ideal sheaf associated with global sections of $L$ reaches its largest possible value. Then we obtain an effective estimate for a weaker version of a theorem of Nakamaye on certain universal Grassmannians. We finally derive from this explicit effective bounds for the Kobayashi and Debarre conjectures.
\end{abstract}

%%%%%%%%%%%%%%%%%%%%%%%%%%%%%%%%%%%%%%%%%%%%%
\section{Introduction}\label{intro}
The celebrated Kobayashi conjecture states that a general hypersurface in $\pb^n$ of sufficient large degree $d\geq d_0$ is Kobayashi hyperbolic. In the last 15 years, at least three important techniques were introduced to study this problem: Siu's slanted vector fields for higher order jet spaces \cite{Siu15}, Demailly's approach for the study of the Green-Griffiths-Lang conjecture through directed varieties strongly of general type (\cite{Dem16}), and Brotbek's recent construction of families of varieties which are deformations of Fermat type hypersurfaces\cite{Bro16}. Inspired by their proof of the Debarre conjecture \cite{BD15} and by the form of equations used by Masuda and Noguchi, Brotbek \cite{Bro16} gave an elegant and relatively short proof of the Kobayashi conjecture. Although his proof does not yield an effective lower bound for the degree $d_0$, he has been able to give a formula for $d_0$ which depends on two constants $M$ and $m_{\infty}$ appearing in connection with certain Noetherianity arguments. In this paper we determine these two constants and thus give an effective lower bound for the degree of general hypersurfaces that are Kobayashi hyperbolic (see Main Theorem \ref{main} below). 

The above mentioned constant $M$, that depends on $n$, arises in the following theorem of Nakamaye: Let $X$ be a K\"ahler manifold and $L$ a big and nef line bundle on $X$. Then the null locus ${\rm Null}(L)$ of $L$ coincides with the augmented base locus $\mathbb{B}_+(L)$. It we fix an ample divisor $A$ on $X$, by the Noetherian property there exists a positive integer $m_0$ such that for $m\geq m_0$ one has ${\rm Bs}(mL-A)=\mathbb{B}_+(L)={\rm Null}(L)$. Now we are led to the following effectivity problem:
\begin{problem}\label{problem}
Can one get an effective bound for $m_0$ that depends on some numerical constants related to $L$ and $A$?
\end{problem}
In general, this problem is hard to answer. Fortunately, in order to compute $M$, we only need to consider some concrete projective manifolds which are universal Grassmannians, and the case where $L$ is the pull back of the tautological line bundle. Our first main result is Theorem \ref{good base locus}, which is an ``almost'' Nakamaye Theorem. To make things more precise, although we could not get the effective bound in the original statement of Nakamaye, we can obtain a bound such that the corresponding base locus is good enough to ensure that the arguments of \cite{Bro16} are still valid. The method can also be used to study the effective problem for the ampleness of the cotangent bundles of general complete intersections (\cite{BD15}).

The constant $m_{\infty}$ appears in relation with the Wronskian ideal sheaf attached to the Demailly-Semple tower. Let $L$ be a very ample line bundle on the K\"ahler manifold $X$. Given a collection of holomorphic sections $s_0,\ldots,s_k\in H^0(X,L)$, Brotbek \cite{Bro16} introduced the Wronskian associated with these sections as
$$
W(s_0,s_1,\ldots,s_k) \in H^0(X_k,\oc_{X_k}({\textstyle\frac{k(k+1)}{2}})\otimes \pi_{0,k}^*(L^{k+1})),
$$
where $X_k$ is the $k$-stage Demailly-Semple tower for the absolute directed variety $(X,T_X)$. Let
$$
\mathbb{W}(X_k,L):={\rm Span}\{W(s_0,\ldots,s_n)|s_0,\ldots,s_n\in H^0(X,L) \}\subset H^0(X_k,\oc_{X_k}({\textstyle\frac{k(k+1)}{2}})\otimes \pi_{0,k}^*(L^{k+1}))
$$
be the associated sublinear system of $H^0(X_k,\oc_{X_k}(\frac{k(k+1)}{2})\otimes \pi_{0,k}^*(L^{k+1}))$. One defines $\mathfrak{w}(X_k,L)$ to be the base ideal of the linear system $\mathbb{W}(X_k,L)$. By Lemma 2.5 in \cite{Bro16} if $L$ is very ample we have
$$
\mathfrak{w}(X_k,L)\subset \mathfrak{w}(X_k,L^2)\subset \ldots \subset \mathfrak{w}(X_k,L^m)\subset\ldots.
$$
Then the Noetherian property shows that this increasing sequence stabilizes for some $m_0:=m_{\infty}(X_k,L)$, and the asymptotic ideal sheaf is denoted by $\mathfrak{m}_{\infty}(X_k,L)$. As was proved in \cite{Bro16}, $\mathfrak{m}_{\infty}(X_k,L)$ does not depend on the very ample line bundle $L$ and is called the \emph{asymptotic Wronskian ideal sheaf}, denoted by $\mathfrak{m}_{\infty}(X_k)$. Now we have the following natural problem:
\begin{problem}\label{wronskian quest}
Determine the constant	$m_{\infty}(X_k,L)$.
\end{problem}

The following theorem solves Problem \ref{wronskian quest}:
\begin{thm}\label{wronskian result}
	If $L$ generates $k$-jets at each point of $X$, then $m_{\infty}(X_k,L)=1$. In particular, if $L$ is known to be only very ample, we have $m_{\infty}(X_k,L)=k$.
\end{thm} 
By Theorem \ref{wronskian result} we have $m_{\infty}=n$ if we take $k=n$. By using some further arguments of \cite{Bro16} we then obtain the following effective lower bound for the degree of general hypersurfaces in the projective manifold $X$ which are Kobayashi hyperbolic:

\begin{main}\label{main}
	Let $X$ be a projective manifold of dimension $n$ and $L$ a very ample line bundle on $X$. Then a generic hypersurface in $|H^0(X,L^d)|$ with $d\geq n^{n+1}(n+1)^{n+2}(n^3+2n^2+2n-1) + n^3 + 3n^2 + 3n$ is Kobayashi hyperbolic. 
\end{main}

The technique can also be used to study the corresponding effectivity question in the Debarre conjecture, and we obtain in this case the following result:
\begin{main}\label{Debarre}
Set $c_0:=\floor{\frac{N+1}{2}}$. For all $d\geq  4c_0(2N-1)^{2c_0+1}+6N-3$, the complete intersection of general hypersurfaces $H_1,\ldots,H_c$ with $H_1,\ldots,H_{c_0}$ in  $\oc_{\pb^N}(d)$ and other hypersurfaces $(H_i)_{i>c_0}$ of any degree, has an ample cotangent bundle.
\end{main}

\section{Asymptotic Wronskian Ideal Sheaf in the Demailly-Semple Tower}\label{amis}

 We denote by $(X_k,V_k)$ the Demailly-Semple tower for $(X,T_X)$, where $X$ is a complex manifold, and $D_j:=P(T_{X_{j-1}/X_{j-2}})$. With an abuse of notation in our context, we also identify $D_j$ with the divisor of $X_k$ given formally by the pull-back $\pi_{j,k}^{-1}(D_j)$. Let $U$ be an open set with coordinates $(z_1,\ldots,z_n)$ such that $\pi_{0,k}^{-1}(U)$ is a trivial product of $U\times \rc_{n,k}$, where $\rc_{n,k}$ is the “universal” rational variety $\cb^{nk}/\mathbb{G}_n$ \cite{Dem95}. We denote by $g_k:\pi_{0,k}^{-1}(U)\rightarrow U\times \rc_{n,k}$ the natural biholomorphism. Then each $D_i\cap \pi_{0,k}^{-1}(U)$ is also a trivial product $U\times E_j$ under the biholomorphism  $g_k$, where $E_j$ is a smooth prime divisor in $\rc_{n,k}$. Let $p\in U$ be the point corresponding to the origin in the coordinate $(z_1,\ldots,z_n)$. Therefore, we have the following biholomorphism $q_k\circ g_k:\pi_{0,k}^{-1}(p)\cap D_j\xrightarrow{\approx} E_j$, where $q_k:U\times \rc_{n,k}\rightarrow \rc_{n,k}$ is the projection to the second factor. In the paper, we will make the convention that any germ of holomorphic curve is defined over $p$, and we will not distinguish the points in $E_j$ and $\pi_{0,k}^{-1}(p)\cap D_j$.

For any germ of holomorphic curve $\nu(t)$ with $\nu(0)=p$, its lift to the $(k-1)$-th Demailly-Semple tower $\nu_{[k-1]}(t)$ induces a section $\nu_{[k-1]}'(t)$ of the line bundle $\nu_{[k]}^*\oc_{X_k}(-1)$. Then by the definition of the Wronskian,  for any $f_0,\ldots,f_n\in \oc_{X,p}$,  $W(f_0,f_1,\ldots,f_k)$ can be seen as a section of $H^0(\rc_{n,k},\oc_{\rc_{n,k}}(m))$ with $m=\frac{k(k+1)}{2}$, given by 
\begin{eqnarray}\label{formula}
W(f_0,\ldots,f_k)(\nu_{[k-1]}'(0)^m)=\begin{vmatrix}
f_0\circ \nu(t) & \ldots & f_k\circ \nu(t) \\
\vdots & \ddots & \vdots \\
\frac{\dm ^k f_0\circ \nu(t)}{\dm t^k} & \ldots & \frac{\dm ^k f_k\circ \nu(t)}{\dm t^k}
\end{vmatrix}_{t=0}.
\end{eqnarray}
We set
$$
S:={\rm Span}\{W(f_0,\ldots,f_n)|f_0,\ldots,f_n\in \oc_{X,p} \}\subset H^0(\rc_{n,k},\oc_{\rc_{n,k}}(m)),
$$
and denote by $\mathcal{I}_{n,k}\subset \oc_{\rc_{n,k}}$ the base ideal of the linear system $S$. By the product property (and the invariance by translation) we find
\begin{equation}
g_k^*\circ q_k^{*}\mathcal{I}_{n,k}=\mathfrak{w}_{\infty}(X_k)
\end{equation}
at $\pi^{-1}_{0,k}(U)$, where $\mathfrak{m}_{\infty}(X_k)$ is the asymptotic Wronskian ideal sheaf introduced in Section \ref{intro}.

\begin{thm}Let $f_0\in m_p^{k+1},f_1,\ldots,f_k\in \oc_{X,p}$, then $W(f_0,\ldots,f_k)=0$ as a section of $H^0(\rc_{n,k},\oc_{\rc_{n,k}}(m))$.
\end{thm}
\begin{proof}
	For any $w\in \rc _{n,k}$, by a theorem in \cite{Dem95}, there exists a germ of curve $\nu:(\cb,0)\rightarrow (X,p)$ such that $\nu_{[k]}(0)=w$ and $\nu_{[k-1]}'(0)\neq 0$. Since $\nu_{[k-1]}'(0)\in \oc_{\rc_{n,k}}(-1)|_{w}$, then if we can show that 
	$W(f_0,\ldots,f_k)(\nu_{[k-1]}'(0)^m)=0$ and take $w$ to be arbitrary, we infer the theorem.
	
	Since $f_0\in m_p^{k+1}$, then the multiplicity of $f_0\circ \nu(t)$ at 0 must be no less than $k+1$, and thus $\frac{(f_0\circ \nu(t))^{(i)}}{t}$ is still holomorphic for $i=0,\ldots,k$. Thus 
	\begin{eqnarray}\nonumber
	W(f_0,\ldots,f_k)(\nu_{[k-1]}'(0)^m)=t\cdot \begin{vmatrix}
	\frac{f_0\circ \nu(t)}{t} & \ldots & f_k\circ \nu(t) \\
	\vdots & \ddots & \vdots \\
	\frac{1}{t}\cdot\frac{\dm ^k f_0\circ \nu(t)}{\dm t^k} & \ldots & \frac{\dm ^k f_k\circ \nu(t)}{\dm t^k}
	\end{vmatrix}_{t=0}=0.
	\end{eqnarray}
\end{proof}

Since the Wronskian operator $W:\oc_{X,p}\times \ldots \times \oc_{X,p}\rightarrow H^0(\rc_{n,k},\oc_{\rc_{n,k}}(m))$ is bilinear, we get
\begin{thmbis}{wronskian result}
	If $L$ generates $k$-jets at each point of $X$, then $m_{\infty}(X_k,L)=1$, where $m_{\infty}(X_k,L)=1$ is the positive integer defined in Section \ref{intro}. In particular, if $L$ is known to be only very ample, we have $m_{\infty}(X_k,L)=k$.
\end{thmbis}

\section{Effectivity of a ``good" base locus for universal Grassmannians}\label{Nakamaye}
In this section we study Problem \ref{problem} for some specific manifolds. First we begin with some definitions. For the reader's convenience, we adopt the same notations as in \cite{Bro16}.

 We consider $V:=H^0(\pb^N,\oc_{\pb^N}(\delta))$, that is, the space of homogenous polynomials of degree $\delta$ in $\cb[z_0,\ldots,z_n]$, and for any $J\subset \{0,\ldots,N\}$, we set
$$
\pb_J:=\{[z_0,\ldots,z_N]\in \pb^N | z_j=0\ {\rm{if}}\ j\in J \}.
$$
Given any $\Delta\in \rm{Gr}_{k+1}(V)$ and $[z]\in \pb^N$, we write $\Delta([z])=0$ if $P(z)=0$ for all $P\in \Delta\subset V$. We then define the family
$$
\mathscr{Y}:=\{(\Delta,[z])\in {\rm{Gr}}_{k+1}(V)\times \pb^N |\Delta([z])=0 \},
$$
and for any $J\subset\{0,\ldots,N\}$, set 
$$
\mathscr{Y}_J:=\mathscr{Y}\cap ({\rm{Gr}}_{k+1}(V)\times \pb_J)\subset \ys\cap {\rm{Gr}}_{k+1}(V)\times \pb^N.
$$
Let $k+1\geq N$, then $q_1:\ys\rightarrow \rm{Gr}_{k+1}(V)$ is a generically finite to one morphism. We also define $q_2:\ys\rightarrow \pb^N$ to be the projection on the second factor. Let $\ls$ be the very ample line bundle on $\grs$ which is the pull back of $\oc(1)$ under the Pl\"ucker embedding. Then $q_1^*\ls|_{\ys_J}$ is a big and nef line bundle on $\ys_J$ for any $J$. For any $J\subset \{0,\ldots,N\}$ we denote by $p_J:\ys_J\rightarrow \grs$, and $\hat{p}_J:\ys_J\rightarrow \pb_J$ the second projection. Similarly we set
$$
E_J:=\{y\in\ys | {\rm{dim}}_y(p_J^{-1}(p_J(y)))>0 \}
$$
$$
G^\infty_J:=p_J(E_J)\subset \grs,
$$
then $E_J=\rm{Null}(q_1^*\ls|_{\ys_J})$. For  $J=\emptyset$ we have $\ys_J=\ys$ and denote by $E:=E_{\emptyset}$ and $G^\infty:=G^\infty_{\emptyset}$.
By the theorem of Nakamaye we have
$$
{\rm{Null}}(q_1^*\ls|_{\ys_J})={\mathbb{B}}_+(q_1^*\ls|_{\ys_J}).
$$
Since $q_1^*\ls\otimes q_2^*\oc_{\pb^N}(1)|_{\ys_J}$ is an ample line bundle on $\ys_J$, by Noetherianity there exists $m_J\in \mathbb{N}$ such that
$$
E_J=\mathbb{B}_+(q_1^*\ls|_{\ys_J})={\rm{Bs}}(q_1^*\ls^{m}\otimes q_2^*\oc_{\pb_J}(-1)|_{\ys_J})
$$ 
for  every $m\geq m_J$.

 From now on we assume that $k+1=N$ (since this is enough in Brotbek's work), and thus $q_1:\ys\rightarrow \rm{Gr}_{k+1}(V)$ is a generically finite-to-one surjective morphism. We first deal with the case $J=\emptyset$.

Let us pick a smooth rational curve $C$ in $\grs$ of degree $1$, namely the curve given by 
$$
\Delta([t_0,t_1]):={\rm{Span}}(z_1^\delta, z_2^\delta,\ldots z_{N-1}^\delta, t_0z_N^\delta+t_1z_0^\delta),
$$
where $[t_0,t_1]\in \pb^1$. We are going to show that $\Delta([t_0,t_1]):\pb^1\rightarrow \grs$ is a smooth embedding from $\pb^1$ to $\grs$, and that the corresponding line $C$ satisfies the degree condition $\ls\cdot C=1$.

Set $\mathbb{I}:=\{I=(i_0,\ldots,i_N)|\;|I|=\delta \}$. Then $\{z^I\}_{I\in \mathbb{I}}$ is a basis for $V$, and the collection of multivectors $w_{(I_0,\ldots,I_k)}$, $I_0<I_1<\ldots< I_k$, is a basis for $\Lambda^{k+1}V$, where $I<J$ is defined by lexicographical order in $\mathbb{I}$. Now we set $J_i\in \mathbb{I}$ to be $z^{J_i}:=z_i^{\delta}$ for $0\leq i\leq k$. Then the curve $\tilde{C}$ defined in $\pb(\Lambda^{k+1}V)$ given by the equations $w_{(I_0,\ldots,I_k)}=0$ for all $(I_0,\ldots,I_k)\neq (J_1,\ldots,J_N) \mbox{
 and } (J_0,\ldots,J_{N-1})$, is of degree 1 with respect to $\ls$. $\tilde{C}$ is totally contained in the image of the Pl\"ucker embedding of $\grs$, whose inverse image is $C$. Now consider the hyperplane $D$ in $\pb^N$ given by  $\{[z_0,\ldots,z_N]|z_0+z_N=0\}$. We have
\begin{lem}\label{intersection}
The intersection number of the curve $q_1^*C$ and the divisor $q_2^*D$ in $\ys$ is $\delta^{N-1}$. Moreover, $q_{1*}q_2^*D\sim \delta^{N-1}\ls$, where {\rm ``$\sim$''} stands for linear equivalence.
\end{lem}
\begin{proof}
	An easy computation shows that $q_1^*C$ and $q_2^*D$ intersect only at the point 
$${\rm{Span}}(z_1^\delta, z_2^\delta,\ldots z_{N-1}^\delta,\\ z_N^\delta+(-1)^{\delta+1}z_0^\delta)\times [1,0\ldots,0,-1]\in \ys$$
with multiplicity $\delta^{N-1}$. The first statement follows. By the projection formula we have
	$$
	q_{1*}q_2^*D\cdot C=q_{1*}(q_2^*D\cdot q_1^*C)=\delta^{N-1}.
	$$
	Since $\rm{Pic(\grs)}\approx \zbb$ with the generator $\ls$, then we get $q_{1*}q_2^*D\sim \delta^{N-1}\ls$ by the fact that $\ls\cdot C=1$.
\end{proof}
We first observe that $q_1^*q_{1*}q_2^*D-q_2^*D$ is an effective divisor  of $\ys$, and by Lemma \ref{intersection} we conclude that $q_1^*\ls^{\delta^{N-1}}\otimes q_2^*\oc_{\pb^N}(-1)$ is effective. We also have a good control of the base locus as follows:
\medskip
\begin{claim}
\begin{eqnarray}\label{claim}
{\rm{Bs}}(q_1^*\ls^{m}\otimes q_2^*\oc_{\pb^N}(-1))\subset q_1^{-1}(G^{\infty})\qquad \mbox{ for any } m\geq\delta^{N-1}.
\end{eqnarray}
\end{claim}
\begin{claimproof}
Indeed, for any $y_0\notin q_1^{-1}(G^{\infty})$, if we denote by $\Delta_0:=q_1(y_0)$, then $ q_1^{-1}(\Delta_0)$ is a finite set, and one can choose a hyperplane $D'\in H^0(\pb^N,\oc_{\pb^N}(1))$ such that $D'\cap q_2\big(q_1^{-1}(\Delta_0)\big)=\emptyset$. From the result above we know that the divisor $q_1^*q_{1*}q_2^*D'-q_2^*D'$ is effective and lies in the linear system $|q_1^*\ls^{\delta^{N-1}}\otimes q_2^*\oc_{\pb^N}(-1)|$ of $\ys$. Now we show that $y_0\notin {\rm{Supp}}(q_1^*q_{1*}q_2^*D'-q_2^*D')$. Indeed, for any $\Delta\in \grs$, if we denote by ${\rm{Int}}(\Delta):=\{[z]\in\pb_N | \Delta([z])=0\}$, then $q_2\big(q_1^{-1}(\Delta_0)\big)={\rm{Int}}(\Delta_0)$. Hence the condition  $D'\cap q_2\big(q_1^{-1}(\Delta_0)\big)=\emptyset$ is equivalent to that ${\rm{Int}}(\Delta_0)\cap D'=\emptyset$. However, for any $\Delta\in q_1\big(q_2^{-1}(D')\big)$, we must have ${\rm{Int}}(\Delta)\cap D'\neq \emptyset$, therefore $\Delta_0\notin q_1\big(q_2^{-1}(D')\big)$ and thus $y_0\notin q_1^{-1}\Big(q_1\big(q_2^{-1}(D')\big)\Big)\supset {\rm{Supp}}(q_1^*q_{1*}q_2^*D'-q_2^*D')$. As $y_0$ was arbitrary, we conclude that
$$
{\rm{Bs}}(q_1^*\ls^{\delta^{N-1}}\otimes q_2^*\oc_{\pb^N}(-1))\subset q_1^{-1}(G^{\infty}).
$$
Since $\ls$ is very ample on $\grs$, and $q_1$ is finite to one morphism outside $q_1^{-1}(G^{\infty})$, we conclude that for any $m\geq \delta^{N-1}$,
$$
{\rm{Bs}}(q_1^*\ls^{m}\otimes q_2^*\oc_{\pb^N}(-1))\subset q_1^{-1}(G^{\infty}).
$$
\end{claimproof}

Now we deal with the general case $J\subset \{0,\ldots,N\}$. First recall our previous notation $p_J:\ys_J\rightarrow \grs$, and let $\hat{p}_J:\ys_J\rightarrow \pb_J$ be the second projection. Our method consists of repeating the above arguments. For any $y_0\notin p_J^{-1}(G_J^{\infty})$, if we denote by $\Delta_0:=p_J(y_0)$, then the set $ p_J^{-1}(\Delta_0)$ is finite. Thus one can choose a generic hyperplane $D\in H^0(\pb^N,\oc_{\pb^N}(1))$ such that $D\cap \pb_J\cap  \hat{p}_J\big(p_J^{-1}(\Delta_0)\big)=\emptyset$, which means that ${\rm{Int}}(\Delta_0)\cap D\cap \pb_J=\emptyset$.
However, for any $\Delta\in p_J\big(\hat{p}_J^{-1}(D\cap \pb_J)\big)$ we must have
${\rm{Int}}(\Delta)\cap D\cap \pb_J\neq\emptyset$, and thus $\Delta_0\notin p_J\big(\hat{p}_J^{-1}(D\cap \pb_J)\big)$, \emph{a fortiori} $y_0\notin p_J^{-1}\Big(p_J\big(\hat{p}_J^{-1}(D\cap \pb_J)\big)\Big)=q_1^{-1}\Big(q_1\big(q_2^{-1}(D)\big)\Big)\cap \ys_J$. Thus the restriction of the effective divisor $q_1^*q_{1*}q_2^*D-q_2^*D$ to $\ys_J$ is well-defined and $y_0\notin {\rm Supp}(q_1^*q_{1*}q_2^*D-q_2^*D|_{\ys_J})$. As $y_0$ was arbitrary, we know that the base locus of the linear system $|q_1^*\ls^{\delta^{N-1}}\otimes q_2^*\oc_{\pb^N}(-1)|_{\ys_J}|$ is contained in $p_J^{-1}(G_J^{\infty})$. Thus by the same argument we have
$$
{\rm{Bs}}(q_1^*\ls^{m}\otimes q_2^*\oc_{\pb^N}(-1)|_{\ys_J})\subset p_J^{-1}(G_J^{\infty})
$$
for $m\geq \delta^{N-1}$.

For any $J\subset \{0,\ldots,N\}$, set $V_J:=H^0(\pb_J,\oc_{\pb_J}(\delta))$. Then there is a natural inclusion $i_J:{\rm Gr}_{k+1}(V_J)\hookrightarrow\grs$, and the pull-back $i_J^*:{\rm Pic}\big(\grs\big)\xrightarrow{\approx} {\rm Pic}\big({\rm Gr}_{k+1}(V_J)\big)$ is an isomorphism between the Picard groups such that $i_J^*\ls$ is still the tautological line bundle on ${\rm Gr}_{k+1}(V_J)$. Moreover, the inclusion $\ys_J\subset \grs\times\pb_J$ factors through ${\rm Gr}_{k+1}(V_J)\times \pb_J$:
\begin{displaymath}
\xymatrix{
	\ys_J \ar[d] \ar[dr]  &  \\
	{\rm Gr}_{k+1}(V_J)\times \pb_J             \ar[r]^{i_J\times \mathds{1}}     & \grs\times\pb_J  }
\end{displaymath}
and thus our above result (\ref{claim}) also holds when $k+1>N$.

In conclusion, we have the following theorem:
\begin{thm}\label{good base locus}
For any $J\subset \{0,\ldots,N\}$, and $k+1\geq N$, we have
$$
{\rm{Bs}}(q_1^*\ls^{m}\otimes q_2^*\oc_{\pb^N}(-1)|_{\ys_J})\subset p_J^{-1}(G_J^{\infty})
$$
for any $m\geq \delta^{N-1}$.
\end{thm}

Since $p_J^{-1}(G_J^{\infty})$ may strictly contain the null locus ${\rm Null}(q_1^*\ls|_{\ys_J})=E_J$,  Theorem \ref{good base locus} does not imply the Nakamaye Theorem used in \cite{Bro16}.  However, as was pointed out to us by Brotbek, this theorem is sufficient to obtain the effective bound appearing in Theorem 3.1 of \cite{Bro16}.  Indeed, the proof given there is still valid under the weaker condition
$$
{\rm{Bs}}(q_1^*\ls^{m}\otimes q_2^*\oc_{\pb^N}(-1)|_{\ys_J})\subset p_J^{-1}(G_J^{\infty})
$$
for any $m\geq M$ and any $J\subset \{0,\ldots,N\}$. By Theorem \ref{good base locus} we can take the constant $M$ to be $\delta^{N-1}$.

The above theorem can be generalized to the case of products of Grassmannians. We then set $N=c(k+1)$, and denote $\mathbf{G}_{k+1}(\delta_i):=\mathbf{G}_{k+1}\Big(|H^0\big(\pb^N,\oc_{\pb^N}(\delta_i)\big)|\Big)$ and $\gf :=\prod_{i=0}^{k+1}\mathbf{G}_{k+1}(\delta_i)$ for simplicity. Let $\ys$ be the universal Grassmannian defined by 
$$
\ys:=\{(\Delta_1,\ldots,\Delta_c,z)\in \gf\times \pb^N | \forall\ i,\Delta_i([z])=0\}.
$$
 Let $p_1:\ys \rightarrow \gf$, $p_2:\ys \rightarrow \pb^N$  and $q_i:\ys\rightarrow \gf_{k+1}(\delta_i)$ be the canonical projections to each factor; then $p_1$ is a generically finite to one morphism. Let $\mathcal{L}$ be the tautological ample line bundle 
$$
\lc:= \oc_{\gf_{k+1}(\delta_1)}(1)\boxtimes \ldots \boxtimes \oc_{\gf_{k+1}(\delta_c)}(1),
$$ and $$\lc(\mathbf{a}):=\oc_{\gf_{k+1}(\delta_1)}(a_1)\boxtimes \ldots \boxtimes \oc_{\gf_{k+1}(\delta_c)}(a_c)$$
for any $c$-tuple of positive integers $\mathbf{a}=(a_1,\ldots,a_c)$.

We then define a smooth rational curve $C_i$ in $\gf$ of degree $1$, given by 
\begin{eqnarray}\nonumber
\Delta([t_0,t_1]):={\rm{Span}}(z_1^{\delta_1},z_{c+1}^{\delta_1}\ldots, z_{kc+1}^{\delta_1})\times {\rm{Span}}(z_2^{\delta_2},z_{c+2}^{\delta_2}\ldots, z_{kc+2}^{\delta_2})\times \ldots\\\nonumber
{}\times {\rm{Span}}(t_1z_{i}^{\delta_i}+t_2z_{0}^{\delta_i},z_{c+i}^{\delta_i},\ldots,z_{kc+i}^{\delta_i})\times \ldots\\\nonumber
{}\times{\rm{Span}}(z_{c}^{\delta_{c}},z_{2c}^{\delta_{c}},\ldots,z_{(k+1)c}^{\delta_{c}}).
\end{eqnarray}
where $[t_0,t_1]\in \pb^1$. It is easy to check that $\Delta([t_1,t_2])$ is a smooth embedding from $\pb^1$ to $\gf$, and that it satisfies the degree condition $\lc(\mathbf{a})\cdot C_i=a_i$ for each $i$. 
Consider the hyperplane $D_i$ of $\pb^n$ given by  $\{[z_0,\ldots,z_N]|z_{i}+z_0=0\}$. Then we have
\begin{lem}
	The intersection number of the curve $p_1^*C_i$ and the divisor $p_2^*D_i$ in $\yc$ is $b_i:=\frac{\prod_{j=1}^{c}\delta_j^{k+1}}{\delta_i}$. Moreover, $p_{1*}p_2^*\oc_{\pb^N}(1)\equiv \ls(\mathbf{b})$, where $\mathbf{b}=(b_1,\ldots,b_c)$.
\end{lem}
\begin{proof}
	It is easy to show that $p_1^*C_i$ and $p_2^*D_i$ intersect only at one point with multiplicity $b_i$. By the projection formula we have
	$$
	p_{1*}p_2^*D_i\cdot C_i=p_{1*}(p_2^*D_i\cdot p_1^*C_i)=b_i.
	$$
	Since $\lc(\mathbf{a}):\mathbf{a}\in\zbb^c\xrightarrow{\sim} {\rm{Pic(\gf)}}$ is an isomorphism, we get  $p_{1*}p_2^*D_i\equiv p_{1*}p_2^*\oc_{\pb^N}(1)\equiv \lc(\mathbf{b})$, thanks to the fact that $\lc\cdot C_i=1$.
\end{proof}
Thus the line bundle $p_1^*\lc(\mathbf{b})\otimes p_2^*\oc_{\pb^N}(-1)$ is effective, and by arguments similar to those of Section \ref{Nakamaye} we also find
\begin{eqnarray}\label{locus debarre}
{\rm Bs}(p_1^*\lc(\mathbf{b})\otimes p_2^*\oc_{\pb^N}(-1))\subset p_1^{-1}(G^{\infty}),
\end{eqnarray}
 where $G^{\infty}$ is the set of points in $\mathbf{G}$ such that their $p_1$-fiber is not a finite set. We can then apply the methods already used above to show that the same estimate also holds for all the strata $\ys_I$ of $\ys$.

\section{Effectivity in the Kobayashi Conjecture}
In this section we give an effective estimate for the Kobayashi conjecture, based mainly on Theorem \ref{good base locus}. First, we briefly recall the ideas of \cite{Bro16} describing how Theorem \ref{good base locus} can be used to give an effective bound for the degrees of general hypersurfaces which are Kobayashi hyperbolic. We try to make our statements self-contained.

Let $X$ be a projective manifold of dimension $n$ and $A$ a very ample line bundle on it. Fix $\tau_0,\ldots,\tau_N\in H^0(X,A)$ in general position. Set $\mathbb{I}:=\{I=(i_0,\ldots,i_N)||I|=\delta \}$. Consider the universal family $\rho:\widetilde{\mathscr{H}}\subset X\times \ab \rightarrow \mathbb{A}$ of hypersurfaces in $X$ given by certain bihomogenous sections of the form
\begin{equation}\label{family}
F(\af):=\sum_{I\in \mathbb{I}}a_I\tau^{(r+k)I},
\end{equation}
where for all $I\in \mathbb{I}$, $a_I\in H^0(X,A^{\epsilon})$, so that $\af:=(a_I)_{I\in \mathbb{I}}\in \mathbb{A}:=\oplus_{I\in \mathbb{I}}H^0(X,A^\epsilon)$. Set $H_\af:=\rho^{-1}(\af)$ and we only consider the smooth locus $\ab_{\rm sm}\subset \ab$ which is the Zariski open subset of $\ab$ such that $H_\af$ is smooth.   Then we can construct a rational map given by the collection of Wronskians
\begin{eqnarray}\nonumber
\Phi: \mathbb{A}\times X_k&\dashrightarrow& P(\Lambda^{k+1}\cb^{\mathbb{I}})\\\nonumber
(\af,w)&\mapsto&([\omega_{I_0,\ldots,I_k}(a_{I_0},\ldots,a_{I_k})(w)])_{I_0,\ldots,I_k\in \ib}.
\end{eqnarray}
The map $\Phi$ factors through the Pl\"ucker embedding
$$
{\rm Pluc}:{\rm Gr}_{k+1}(\cb^{\ib})\hookrightarrow P(\Lambda^{k+1}\cb^{\mathbb{I}}).
$$

If we denote by $\hat{X}_k$  the blow-up $\nu_k:\hat{X}_k\rightarrow X_k$ of the asymptotic $k$-th Wronskian ideal sheaf  $\mathfrak{w}_{\infty}(X_k)$ of the $k$-th Demailly-Semple tower of $X$, such that $\nu_{k}^{*}\big(\mathfrak{w}_{\infty}(X_k)\big)=\oc_{\hat{X}_k}(-F)$ for some effective divisor $F$.  Then $\nu_k$ partially resolves the singularities of $\Phi$, under the condition (based on Proposition 3.8 in \cite{Bro16}) that
 $$\epsilon\geq m_{\infty}(X,A),\quad {\rm dim}\hat{X}_k+k<\#\mathbb{I}_x$$
 for any $x$ in $X$. Here
 $$
 \ib_x:=\{I\in\ib,|\tau^{I}(x)\neq 0\}.
 $$
  Since for any $x$ we have $\#\mathbb{I}_x\geq\binom{N-n+\delta}{\delta}$, then if
\begin{equation}\label{condition 1}
(k+1)n<\binom{N-n+\delta}{\delta},\quad \epsilon\geq m_{\infty}(X,A),\tag{$\bigstar$}
\end{equation}
 there exists a non-empty Zariski open subset $\ab_{\rm def}\subset \ab_{\rm sm}$ such that the rational map $\Phi|_{\ab_{\rm def}\times X_k}\dashrightarrow \text{Gr}_{k+1}(\cb^{\ib})$ can be resolved into a regular morphism
$$
\hat{\Phi}:\ab_{\rm def}\times \hat{X}_k\rightarrow \text{Gr}_{k+1}(\cb^{\ib}).
$$

For example, we can take $N=n+1, k=n$ and $\delta=n(n+1)$ to satisfy Condition \ref{condition 1}.

 Since the blow-up of the Wronskian ideal sheaf behaves well under restriction, \emph{i.e.} if $Y\subset X$ is a smooth subvariety, then $\nu_k:\nu_k^{-1}(Y_k)\rightarrow Y_k$ is also a blow-up of the Wronskian ideal sheaf $\mathfrak{w}_{\infty}(Y_k)$, and thus it behaves well in families (ref. Proposition 2.7 in \cite{Bro16}). Therefore, if we denote by $\mathcal{\hs}_k^{\rm rel}$ the relative $k$-th Demailly-Semple tower of the smooth family $\hs\subset \ab_{\rm def}\times X$ given by (\ref{family}), $\nu_k$ will induce a morphism $\hat{\nu}_k:\hat{\hs}_k^{\rm rel}\subset \ab_{\rm def}\times \hat{X}_k\rightarrow \hs_k^{\rm rel}$, such that for any $\af\in \ab_{\rm def}$, if we denote $\hat{H}_{\af,k}:=\hat{\nu}_k^{-1}(H_{\af,k})$, then $\hat{\nu}_k:\hat{H}_{\af,k}\rightarrow H_{\af,k}$ is the blow-up of the Wronskian ideal sheaf $\mathfrak{w}_{\infty}(H_{\af,k})$.

If we introduce the map induced by $\hat{\Phi}$
\begin{eqnarray}\nonumber
\hat{\Psi}:\ab_{\rm def}\times \hat{X}_k&\rightarrow& \rm{Gr}_{k+1}(\cb^{\ib})\times \pb^N\\\nonumber
(\af,\hat{w})&\mapsto&\big(\hat{\Phi}(\af,w),[\tau^r(\hat{w})]\big),
\end{eqnarray}
then when restricted to $\hat{\hs}_k^{\rm rel}$ the morphism $\hat{\Psi}$ factors through $\ys$, which is the universal Grassmannian defined in Section \ref{Nakamaye}
$$
\mathscr{Y}:=\{(\Delta,[z])\in {\rm{Gr}}_{k+1}(\cb^{\ib})\times \pb^N |\Delta([z])=0 \}.
$$
By the definition of $\hat{\Psi}$, for any positive integer $m$ we have
$$
\hat{\Psi}^{*}(q_1^*\ls^m\otimes q_2^*\oc_{\pb^N}(-1))=\nu_k^*(\oc_{X_k}(mk')\otimes \pi_{0,k}^*A^{m(k+1)(\epsilon+k\delta)-r})\otimes\oc_{\hat{X}_k}(-mF),
$$
where $k':=\frac{k(k+1)}{2}$.

Now take $k+1=N$. By Theorem \ref{good base locus}, for any $m\geq M:=\delta^{N-1}$,  we have a good control of the base locus
$$
{\rm Bs}(q_1^*\ls^m\otimes q_2^*\oc_{\pb^N}(-1))\subset q_1^{-1}(G^{\infty}),
$$
where $G^{\infty}$ is the set of points in ${\rm{Gr}}_{k+1}(\cb^{\ib})$ such that the inverse image of $q_1$ is not finite.

If we can find a non-empty Zariski open subset $\ab_{\rm nef}\subset \ab_{\rm def}$ such that 
\begin{eqnarray}\label{avoid}
\hat{\Phi}^{-1}(G^{\infty})\cap (\ab_{\rm nef}\times \hat{X}_k)=\emptyset,
\end{eqnarray}
then we can prove that
\begin{equation}\label{pull back}
\nu_k^*(\oc_{X_k}(mk')\otimes \pi_{0,k}^*A^{m(k+1)(\epsilon+k\delta)-r})\otimes\oc_{\hat{X}_k}(-mF)|_{\hat{H}_{\af,k}}
\end{equation}
is nef on $\hat{H}_{\af,k}$ for any $\af\in \ab_{\rm nef}$. If one applies the Fundamental Vanishing Theorem (ref. \cite{Dem95}) to constrain all entire curves in the base locus, the bundles in (\ref{pull back}) need to be twisted by the opposite of an ample divisor, \emph{i.e.} we must have
\begin{equation}\label{condition 2}
m(k+1)(\epsilon+k\delta)<r.\tag{$\blacklozenge$}
\end{equation}

In order to avoid the positive dimensional fibers, as expressed by condition (\ref{avoid}), and as was proved in Lemma 3.8 of \cite{Bro16}, $\delta$ needs to satisfy the following condition:
\begin{equation}\label{codim}
\delta+1>{\rm dim}(\hat{X}_k)=n+(n-1)k.
\end{equation}
If we fix $N=k+1=n+1$,   $\delta$ only needs to satisfy
\begin{equation}\label{condition 3}
\delta\geq n^2.\tag{$\spadesuit$}
\end{equation}

Let $X_{n}$ be the $n$-th Demailly-Semple tower for $X$. By Theorem \ref{wronskian result} we have
$$
\mathfrak{w}_{\infty}(X_{n},A)=\mathfrak{w}(X_{n},A^{n}),
$$
which implies $m_{\infty}(X_{n},A)=n$. By the constructions of $H_{\af}$, we see that 
$$
H_{\af}\in H^0\Big(X,\oc_X\big(\epsilon+\delta(r+k)\big)A\Big).
$$
By Conditions \ref{condition 1}, \ref{condition 2} and \ref{condition 3}, for any $n\geq 2$, we can take $N=k+1=n+1$ and $\delta=n(n+1)$, to conclude that:
\begin{thm}\label{maindecom}
	Let $X$ be a projective manifold endowed with a very ample line bundle, for any degree $d\in \mathbb{N}$ satisfying
\begin{eqnarray}\label{decompose}\nonumber
	\exists \epsilon\geq n,\ r>\delta^{N-1}(k+1)(\epsilon+k\delta)=n^n(n+1)^{n+1}\big(\epsilon+n^2(n+1)\big):\\
	d=n(n+1)(r+n)+\epsilon,
\end{eqnarray}
the general hypersurfaces in $|H^0(X,\oc_X(dA))|$ is Kobayashi hyperbolic. 
\end{thm}
In particular, for any $d\geq d_0:=n^{n+1}(n+1)^{n+2}(n^3+2n^2+2n-1) + n^3 + 3n^2 + 3n$, there exists a pair $(r,\epsilon)\in \mathbb{N}^2$ satisfying Condition \ref{decompose} in Theorem \ref{maindecom}. This proves our Main Theorem \ref{main}.
\begin{rem} By using Siu's technique of slanted vector fields on higher jet spaces and combining it with techniques of Demailly, the first effective lower bound for the degree of the general hypersurface which is weakly hyperbolic was given by Diverio, Merker and Rousseau {\rm\cite{DMR10}}, where they confirmed the Green-Griffiths-Lang conjecture for generic hypersurfaces in $\pb^n$ of degree $d\geq 2^{(n-1)^5}$. Later on, by means of a very elegant combination of the holomorphic Morse inequalities and a probabilistic interpretation of higher order jets, Demailly was able to improve the lower bound to $d\geq\floor[\Big]{\frac{n^4}{3}\Big(n\log\big(n\log(24n)\big)\Big)^n}$ {\rm\cite{Dem10}}. The latest best known bound was $d\geq (5n)^2n^n$ by Darondeau {\rm\cite{Dar15}}. 
\end{rem}

\section{Effectivity for the Ampleness of the Complete Intersection}

 Let $M$ be a projective manifold of dimension $N$ and $A$ a very ample line bundle on it. For any integer $c\geq \frac{N}{2}$, by the work \cite{BD15} and \cite{Xie16} we know that when $m$ is large enough, the complete intersection $X:=H_1\cap \ldots\cap H_c$ of general hypersurfaces in $H^0(X,mA)$ has ample cotangent bundle. In \cite{Xie16} Songyan Xie gave an effective bound $N^{N^2}$ for $m$. Our goal in this section is to obtain a better estimate, based on the work \cite{BD15}.
 
To start with, we will outline the stategy of \cite{BD15} to prove that the general complete intersection has an ample cotangent bundle, and show what is involved in order to obtain the effectivity. The general ideas are similar to those used in the proof of Kobayashi hyperbolicity. Howvever, in that case, one does not encounter the intrinsic singularities related with the Demailly-Semple towers, since the Wronskian ideal sheaf arises only for jets of order $2$ and higher. 
 
  First consider a family $p:\mathcal{X}\rightarrow S$ of complete intersection subvarieties in $M$ cut out by  certain $c$-bihomogenous sections, depending on two $c$-tuples of positive integers $\bm{\epsilon}:=(\epsilon_1,\ldots,\epsilon_c)$ and $\bm{\delta}:=(\delta_1,\ldots,\delta_c)$, and a positive integer $r$ fixed later. Then for general $\bm{\lambda}\in S$, the fiber $X_{\lambda}:=p^{-1}(\bm{\lambda})$ is a compete intersection of $c$ smooth hypersurfaces $H_1\in |H^0\big(M,\oc_M(d_1A)\big)|,\ldots,H_c\in |H^0\big(M,\oc_M(d_cA)\big)|$, where $d_p:=\epsilon_p+(r+1)\delta_p$ for each $p=1,\ldots,c$.  Now we can define a rational map $\Psi$ from  $S\times \pb(\Omega_M)$ to $\gf_2(\delta_1)\times \ldots \gf_2(\delta_c)\times \pb^{2c}$. As was proved in Proposition 2.4 in \cite{BD15}, as soon as $\delta_p\geq2$ for any $1\leq p\leq c$, after shrinking the parameter space $S$ to some non-empty Zariski open set $U_{\rm def}$, when restricted to $U_{\rm def}\times \pb(\Omega_M)$, the rational map $\Psi$ is a regular morphism.
 
 Along $\pb(\Omega_{\mathcal{X}/S})\cap \big(U_{\rm def}\times \pb(\Omega_M)\big)$, the map $\Psi$ factors through the universal Grassmannian $\ys$ defined by
 $$
 \ys:=\{(\Delta_1,\ldots,\Delta_c,z)\in \gf_2(\delta_1)\times \ldots \gf_2(\delta_c)\times \pb^{2c} | \forall\ i,\Delta_i([z])=0\}.
 $$ 
 Here we use the notations in Section \ref{Nakamaye}. By the definition of $\Psi$, for any $c$-tuple of positive integers $\bm{\af}$, we have
 $$
 \Psi|_{\pb(\Omega_{\mathcal{X}/S})\cap \big(U_{\rm def}\times \pb(\Omega_M)\big)}^{*}\big(\ls(\af)\boxtimes \oc_{\pb^{2c}}(-1)\big)=\oc_{\pb(\Omega_{\mathcal{X}/S})}(\sum_{i=1}^{c}a_i)\otimes \pi^*\oc_X \Big(\big(2\sum_{p=1}^{c}(\epsilon_p+\delta_p)a_p-r\big)A\Big),
 $$
 where $\pi:\pb(\Omega_{\mathcal{X}/S})\rightarrow X$ is the projection map, and $\lc(\mathbf{a}):=\oc_{\gf_{2}(\delta_1)}(a_1)\boxtimes \ldots \boxtimes \oc_{\gf_{2}(\delta_c)}(a_c)$. Observe that if we take 
 \begin{equation}\label{debarre condition 1}
 \tag{$\heartsuit$}
 2\sum_{p=1}^{c}(\epsilon_p+\delta_p)b_p<r,
 \end{equation}
 then the above bundle becomes the symmetric differential forms with a negative twist.
 By (\ref{locus debarre}), if we set $b_i:=\frac{\prod_{i=j}^{c}\delta_j^{2c}}{\delta_i}$, then we have
 \begin{equation}
 {\rm Bs}\big(p_1^*\lc(\mathbf{b})\otimes p_2^*\oc_{\pb^{2c}}(-1)\big)\subset G^{\infty},
 \end{equation}
 where $G^{\infty}$ to be the set of points in $\mathbf{G}:=\gf_2(\delta_1)\times \ldots \gf_2(\delta_c)$ such that the fiber of projection map $p_1:\ys\rightarrow \gf$ is not a finite set. Thus if we can shrink $U_{\rm def}$ to some non-empty Zariski open subset $U_{\rm nef}\subset U_{\rm def}$ such that 
 \begin{eqnarray}\label{avoid exceptional debarre}
 \Psi\Big(\pb(\Omega_{\mathcal{X}/S})\cap\big(U_{\rm nef}\times \pb(\Omega_M)\big)\Big)\cap p_1^{-1}(G^\infty)=\emptyset,
 \end{eqnarray}
 then the bundle
 \begin{equation}
 \Psi|_{\pb(\Omega_{\mathcal{X}/S})\cap \big(U_{\rm nef}\times \pb(\Omega_M)\big)}^{*}\big(\ls(\bff)\boxtimes \oc_{\pb^{2c}}(-1)\big)\big|_{\pb(\Omega_{X_{\bm{\lambda}}})}=\oc_{\pb(\Omega_{X_{\bm{\lambda}}})}(\sum_{i=1}^{c}a_i)\otimes \pi^*\oc_X\big(-q(\bm{\epsilon},\bm{\delta},r)A\big)
 \end{equation}
 is nef when restricted to $\pb(\Omega_{X_{\bm{\lambda}}})$ for any $\bm{\lambda}\in U_{\rm nef}$. Here we set $q(\bm{\epsilon},\bm{\delta},r):=r-2\sum_{p=1}^{c}(\epsilon_p+\delta_p)b_p$. Thus if Condition \ref{debarre condition 1} is fulfilled, for any $\bm{\lambda}\in U_{\rm nef}$, $X_{\bm{\lambda}}$ has an ample cotangent bundle. 
 
In order to avoid the exceptional locus on which Condition \ref{avoid exceptional debarre} holds,  as was shown in the proof of Lemma 2.6 in \cite{BD15}, $\bm{\delta}$ needs to satisfy 
\begin{equation}\label{debarre condition 2}
\min_{i=1,\ldots c}\delta_i+1>{\rm dim}\big(\pb(\Omega_M)\big)=2N-1.\tag{$\diamondsuit$}
\end{equation}
 
 In conclusion,  combining Condition \ref{debarre condition 1} and \ref{debarre condition 2} together, we have a more refined result compared to Corollary 2.8 in \cite{BD15}:
\begin{cor}\label{condition vector}
	On a $N$-dimensional smooth projective variety $M$, equipped with a very ample line bundle $A$, for any degrees $(d_1,\ldots,d_c)\in (\mathbb{N})^{c}$ satisfying
	\begin{eqnarray}\nonumber
	\exists \bm{\delta}_{(\delta_p\geq 2N-1)},\ \exists \bm{\epsilon}_{(\epsilon_p\geq 1)},\ \exists r>2c\prod_{i=1}^{c}\delta_i^2+2\sum_{i=1}^{c}\frac{\epsilon_i\prod_{j=1}^{c}\delta_j^2}{\delta_i}, \mbox{ s.t. }\\\nonumber
	d_p=\delta_p(r+1)+\epsilon_p \qquad (p=1,\ldots,c),
	\end{eqnarray}
	the complete intersection $X:=H_1\cap \ldots H_c$ of general hypersurfaces $H_1\in |A^{d_1}|,\ldots,H_c\in |A^{d_c}|$ has an ample cotangent bundle.
\end{cor}

Since any subvariety of a variety with ample cotangent bundle also has ample cotangent bundle, we only need to deal with the case $c_0:=\floor{\frac{N+1}{2}}$. If we take $\delta_1=\ldots=\delta_{c_0}=2N-1$, then any $d\geq 4c_0(2N-1)^{2c_0+1}+6N-3,$ can be decomposed as
$$
d=(2N-1)(r+1)+\epsilon
$$
with $1\leq\epsilon< 2N$ and $r>4c_0(2N-1)^{2c_0}$, satisfying the conditions in Corollary \ref{condition vector}. In particular, if $M$ is taken to be $\pb^N$, then we can take $A$ to be $\oc_{\pb^N}(1)$, and thus for all $d\geq 4c_0(2N-1)^{2c_0+1}+6N-3$, the complete intersection of general hypersurfaces $H_1,\ldots,H_c$ with $H_1,\ldots,H_{c_0}$ in  $\oc_{\pb^N}(d)$ and other hypersurfaces $(H_i)_{i>c_0}$ of any degree has an ample cotangent bundle. Our Main Theorem \ref{Debarre} is proved.
\medskip

\noindent
\textbf{Acknowledgements.}
I would like to warmly thank my thesis supervisor Professor Jean-Pierre Demailly for his
very fruitful discussions and suggestions, and also for his patience
and disponibility. The effective bounds in the work were improved substantially during the ``K\"ahler geometry Workshop" held at Institut Fourier in May 2016, under the auspices of the ALKAGE project. A thousand thanks to Damian Brotbek for his very careful reading of the manuscript, useful suggestions and comments, to Songyan Xie for explaining his work and ideas patiently, and also to Lionel Darondeau for his interests and remarks on the work.

\end{document}